\documentclass[10pt]{amsart}

\numberwithin{equation}{section} 

\usepackage{fancyhdr}
\usepackage{bm}
\usepackage{amssymb}
\usepackage{xypic}

\newtheorem{thm}{Theorem}[section]
\newtheorem{cor}[thm]{Corollary}
\newtheorem{lem}[thm]{Lemma}
\newtheorem{prop}[thm]{Proposition}
\theoremstyle{definition}	
\newtheorem{defn}[thm]{Definition}
\newtheorem{rem}[thm]{Remark}
\newtheorem{const}[thm]{Construction}
\newtheorem{exam}[thm]{Example}

\newcommand{\bracketsix}[6]{{{\left\{ \begin{array}{ll}
				{{#1}}&{{#2}}\\
				{{#3}}&{{#4}}\\
			          {{#5}}&{{#6}}\end{array}\right.}}}
\def\defeq{\overset{\mathrm{def}}=}

\def\goth{\mathfrak}

\def\a{\alpha}
\def\b{\beta}

\def\Ga{\Gamma}

\def\cC{\mathcal{C}}

\def\cK{\mathcal{K}}
\def\cL{\mathcal{L}}

\def\cO{\mathcal{O}}

\def\cS{\mathcal{S}}

\def\FF{\mathbb{F}}
\def\GG{\mathbb{G}}

\def\SS{\mathbb{S}}
\def\W{\mathbb{W}}
\def\Z{\mathbb{Z}}
\def\ZZ{{{\Z}}}

\def\Aut{\mathrm{Aut}}
\def\End{{{\mathrm{End}}}}
\def\Ext{\mathrm{Ext}}
\def\Gal{\mathrm{Gal}}
\def\Hom{\mathrm{Hom}}

\def\lim{\mathrm{lim}}

\def\Pic{{\mathrm{Pic}}}
\def\longr{{{\longrightarrow\ }}}
\def\alg{{{\mathrm{alg}}}}

\def\Picalg{{{(\Pic_n)_{\alg}}}}
\def\PicalgO{{{(\Pic_n)_{\alg}^0}}}
\def\Gl{{{\mathrm{Gl}}}}
\def\Laa{{{\cL}}}
\def\Ka{{{\cK}}}
\def\Et{{{(E_2)}}}

\def\Zdet{{{S^0\langle\det\rangle}}}

\def\map{{{\mathrm{map}}}}
\def\det{{{\mathrm{det}}}}
\def\SGG{{{\mathrm{S}\GG}}}

\newcommand{\f}{\mathbb{F}_{p^n}}

\newcommand{\G}{\mathbb{G}}
\newcommand{\s}{\mathbb{S}}

\date{\today} 
\begin{document}

\title{On Hopkins' Picard groups for the prime $3$ and 
chromatic level $2$} 

\begin{abstract}
We give a calculation of Picard groups $Pic_2$  
of $K(2)$-local invertible spectra and $Pic(\cL_2)$ 
of $E(2)$-local invertible spectra, both at the prime $3$.  The
main contribution of this paper is to calculation the subgroup
$\kappa_2$ of invertible spectra $X$ with $(E_2)_\ast X
\cong (E_2)_\ast S^0$ as twisted modules over the
Morava stabilizer group $\GG_2$.
\end{abstract}
 
\author{Paul Goerss, Hans-Werner Henn, Mark Mahowald
and Charles Rezk}
\thanks{The first and fourth authors were partially supported by the National Science Foundation. The second author was partially supported by ANR ``HGRT''}

\address{Department of Mathematics, 
Northwestern University, Evanston, IL 60208, U.S.A.}  
\address{Institut de Recherche Math\'ematique Avanc\'ee,
C.N.R.S. - Universit\'e de Strasbourg, F-67084 Strasbourg,
France}
\address{Department of Mathematics, 
Northwestern University, Evanston, IL 60208, U.S.A.}  
\address{Department of Mathematics, 
University of Illinois at Urbana-Champaign, Urbana Champaign, IL 61801, 
U.S.A.}
 
\maketitle 

\section{Introduction}   

\medskip

Let $\cC$ be a symmetric monoidal category with product $\wedge$ and
unit object $I$.  An object $X\in \cC$ is {\it invertible} if
there exists an object $Y$ and an isomorphism $X\wedge Y\cong I$. If the
collection of isomorphism classes of invertible objects is a set,
then $\wedge$ defines a group structure on this set. This group is
called the  {\it Picard group} of $\cC$ and is denoted by $Pic(\cC)$. 
It is a basic invariant of the category $\cC$.

An immediate example from stable homotopy is the category $\cS$ of
spectra and homotopy classes of maps. However,
the answer turns out to be unsurprising: the only invertible spectra
are the sphere spectra $S^n$, $n \in \ZZ$ and the map $\ZZ \to \Pic(\cS)$
sending $n$ to $S^n$ is an isomorphism. See \cite{Picard}.

An insight due to Mike Hopkins is that the problem becomes considerably
more interesting if we pass to the localized stable homotopy categories
which arise in chromatic stable homotopy theory; that is, if we localize
the stable category with respect to complex orientable cohomology
theories.  We fix a prime $p$ and work with $p$-local
theories -- for example, the Morava $K$-theories $K(n)$, $n \geq 1$,
or the closely related Johnson-Wilson theories $E(n)$.
We will write $L_E$ for localization
with respect to $E$; however it is customary to abbreviate $L_{E(n)}$
as $L_n$. We will also write $\cK_n$ for the homotopy category of
$K(n)$-local spectra and $\cL_n$ for the category of $E(n)$-local
spectra. See \cite{HS} for
a great deal of information about the basic structure of these categories.

If $E$ is any homology theory, then the homotopy
category of $E$-local spectra has a symmetric monoidal structure
given by smash product with unit $L_ES^0$.
The main project of this paper is to study the Picard group of the 
categories $\cK_2$ and $\cL_2$ at the prime $3$. The Picard
groups of $\cK_1$ and $\cL_1$ were investigated in \cite{HMS}.
A feature of that paper is the observation
that $\Pic(\cL_1)$ at $p=2$
had an exotic element not immediately detected by algebraic methods;
we will investigate similar phenomena.

There are two steps to the calculation of the Picard group of these
localized stable homotopy categories. The first is algebraic.
If $E_\ast$ is a homology theory with $E_\ast E$ flat over $E_\ast$,
we denote by $\Pic(E)_\alg$ the Picard group of 
invertible $E_\ast E$-comodules. For
$\cL_2$ and $\cK_2$ this group is known and recorded
in Section \ref{alg-pic-sec} below. In both of these cases, there
is then a homomorphism from the  Picard group of the $E$-local
stable homotopy category 
$$
\Pic(\cL_E) \longr \Pic(E)_\alg
$$
sending $X$ to $E_\ast X$. In our examples this will be onto and the
interest then turns to the kernel. This is the group of {\it exotic} elements:
those $E$-local invertible spectra so that $E_\ast X \cong E_\ast$ 
as $E_\ast E$ comodules. For $\cK_2$ and $\cL_2$ for $p > 3$,
there are no non-trivial exotic elements. 
Write $\kappa_2$ and $\kappa(\cL_2)$ for the group of
exotic elements in $\cK_2$ and $\cL_2$ respectively at $p=3$.
Kamiya and Shimomura \cite{shiPic} have shown that
$\kappa(\cL_2)$ is
either of order $3$ or order $9$; in particular, it is known to be
non-trivial. Our main result shows that it is of order $9$.

\begin{thm}\label{main-thm} Let $\kappa_2$ be the group of
exotic elements in the $K(2)$-local Picard group $\Pic(\cK_2)$ at
$p=3$. Then there is an isomorphism
$$
\kappa_2 \cong \Z/3 \times \Z/3.
$$
Furthermore the localization map
$$
\kappa(\cL_2) \longr \kappa_2
$$
is an isomorphism.
\end{thm}

The first part of this result is proved in section 5; the second in section 6.
Along the way
we will take the opportunity to revisit some of the standard results of $K(2)$-local
homotopy theory; for example, in section 4, we will show that the $K(2)$-local
Adams-Novikov Spectral Sequence for any exotic element of the Picard
group has a global vanishing line
at the $E_\infty$-page, although it has no such line at $E_2$.

Combining the algebraic calculation with Theorem \ref{main-thm} now the completes
the calculation for $p=3$. Let $\ZZ_p$ denote the $p$-adic integers.

\begin{thm}\label{total-pic} There are isomorphisms
$$
\Pic(\cK_2) \cong \Z_3 \times \Z_3 \times \Z/16 \times \Z/3 \times \Z/3.
$$
and
$$
\Pic(\cL_2) \cong \ZZ \times  \Z/3 \times \Z/3.
$$
\end{thm}

For more precise statements and details see Theorem \ref{pnot2} and Proposition \ref{pic-l-n}.

It is worth emphasizing that the two summands of $\kappa_2$ do not
have equal dignity. At $p=2$ and $n=1$, the exotic element of
$\kappa_1$ can be detected by real $K$-theory; that is, there is
an element of the Picard group so that $K_\ast X \cong K_\ast S^0$,
but $KO_\ast X \not\cong KO_\ast S^0$. In the same way, at $p=3$ and
$n=2$, there are elements of the Picard group so that $\Et_\ast X
\cong \Et_\ast S^0$, but for which $(E_2^{hG})_\ast X \not\cong
(E_2^{hG})_\ast S^0$ for some finite group $G$ acting on $E_2$. However,
there are also non-trivial elements of $\kappa_2$ so that 
$$
(E_2^{hG})_\ast X \cong (E_2^{hG})_\ast S^0
$$
for all finite groups acting on $E_2$. Detecting and constructing
these truly exotic elements is the novel part of our argument. 

\section{Background}

In this section we work on the basic technology and known results
behind our work. Recall we have a fixed prime $p$, which will 
eventually be $3$.

\subsection{Morava $K$-theory and related homology 
theories}\label{E-theory-sec}

We will work with $p$-local $2$-periodic complex oriented
cohomology theories $E$ equipped with a $p$-typical formal group law
$F$. Thus $E_\ast \cong E_0[u^{\pm 1}]$ where $u$ is a unit in degree
$-2$ and the associated formal group law $F$ is determined by the
$p$-series
$$
[p]_F(x) = px +_F v_1x^p +_F v_2x^{p2}+_F\cdots.
$$
where $v_i$ is in degree $2(p^i-1)$.

Basic examples of such theories are the Morava $K$-theories
$K(n)_\ast(-)$ with
$$
K(n)_\ast = \FF_p[u^{\pm 1}]
$$
where $u$ is in degree $2$; the formal group law is
the Honda formal group law $\Ga_n$ of height $n$ with
$p$-series $[p]_{\Ga_n}(x) = u^{-(p^n-1)}x^{p^n}$.

There is an associated Landweber exact homology theory
$(E_n)_\ast$ -- the Morava $E$-theory, also known as the 
Lubin-Tate theory. It is customary to 
choose the $p$-typical formal group law over $(E_n)_\ast$
with
\begin{equation}\label{LT-vi}
v_i = \bracketsix{u^{-(p^i-1)}u_i,}{1 \leq i < n;}{u^{-(p^n-1)},}{i=n;}{0,}{i > n.}
\end{equation}
This equation defines elements $u_i$ in $(E_n)_0$ and 
an isomorphism
$$
W(\FF_{p^n})[[u_1,\cdots,u_{n-1}]] \cong (E_n)_0
$$
where $W(-)$ is the Witt vector functor. The associated formal group law
is the universal deformation of the Honda formal group law
over $\FF_{p^n} \otimes K(n)_\ast$. 

Define the (big) {\it Morava stabilizer group} $\GG_n$ to be 
the automorphisms of the pair $(\FF_{p^n},\Ga_n)$.
Since $\Ga_n$ is defined over $\FF_p$, we have
$$
\GG_n = \s_n \rtimes \Gal(\f/\FF_p)
$$
where $\s_n$ is the automorphisms of $\Ga_n$ over $\FF_{p^n}$.

The group $\GG_n$ acts on the deformations of $(\FF_{p^n},\GG_n)$
and hence on $(E_n)_\ast$. By the Hopkins-Miller Theorem
\cite{GoH} this action can be lifted uniquely to
$E_\infty$-ring spectra; that is, the group $\GG_n$ acts
on the spectrum $E_n$ through $E_\infty$-ring maps.
In \cite{DH}, Devinatz and Hopkins showed that for all 
closed subgroups $K \subseteq \GG_n$ there is a homotopy
fixed point spectrum $E_n^{hK}$ and a spectral sequence
\begin{equation}\label{descent-ss}
H^s(K,(E_n)_t) \Longrightarrow \pi_{t-s}E_n^{hK}.
\end{equation}
Here the group cohomology is continuous cohomology. If $K$ 
is finite, this is the usual fixed point spectral sequence. By contrast,
if $K=\GG_n$ itself, then $E_n^{hK} \simeq L_{K(n)}S^0$ and
this is the Adams-Novikov Spectral Sequence. Compare
(\ref{ANSS}) below.

The $2$-periodic Johnson-Wilson theories $E(n)$ have 
$$
E(n)_0 \cong \ZZ_{(p)}[u_1,\ldots,u_{n-1},u_n^{\pm 1}]
$$
and $p$-typical formal group law defined by the Equation
(\ref{LT-vi}). There is a morphism of homotopy commutative ring spectra
$E(n) \to E_n$ which, on coefficients, is the map indicated by the
notation. The associated homology theories
have the same acyclics and induce the same localization
functor, which we will call $L_n$, but note that the notation
$(E_n)_\ast(-)$ does {\bf not} mean the homology theory
associated to $E_n$.

\begin{rem}[{\bf Local smash products}]\label{E-smash}
If $X$ and $Y$ are $E$-local, then $X \wedge Y$ need not be
$E$-local; hence the natural smash product for $E$-local
spectra is
$$
X \wedge_{E} Y = L_{E} (X \wedge Y).
$$
The notation $\wedge_{E}$ is very cumbersome, however, so we
will drop it when it is understood. In particular, throughout most
of this paper, $X$ and $Y$ are $K(n)$-local spectra, and
$X \wedge Y$ means $X \wedge_{K(n)}Y$.
\end{rem}

From this point of view, the natural definition of $(E_n)_\ast(-)$ is 
$$
(E_n)_*X \defeq \pi_*L_{K(n)}(E_n\wedge X).
$$
Under appropriate circumstances (see \cite{HS}) $(E_n)_\ast X$ is a
completion of $\pi_\ast (E_n \wedge X$); however, $(E_n)_\ast(-)$ is not 
a homology theory because it does not take wedges to sums.
Since $\GG_n$ acts on $E_n$, $\GG_n$ acts on $(E_n)_\ast X$.
The $(E_n)_\ast$-module $(E_n)_\ast X$ is equipped with the
${\goth m}$-adic topology 
where ${\goth m}$ is the maximal ideal in $(E_n)_0$. This topology 
is always topologically complete, but need not be separated. 
With respect to this topology, the group $\GG_n$ acts through
continuous maps and the action is twisted because it
is compatible with the action of $\GG_n$ on the coefficient ring
$E_\ast$. See \cite{GHMR} \S 2 for
some precise assumptions which guarantee that $(E_n)_\ast X$ 
is complete and separated. We will call these modules twisted 
$\GG_n$-modules.

Let  $\map (\GG_n,(E_n)_\ast)$ be the $(E_n)_\ast$-module of
continuous maps out of $\GG_n$. We give this the diagonal
$\GG_n$-action; that is, 
\begin{equation}\label{diagonal}
(g\phi)(x) = g\phi(g^{-1}x).
\end{equation}
Then (by \cite{StricklandGH} among other sources) there is
an isomorphism of twisted $\GG_n$-modules
\begin{equation}\label{e-hom-e}
(E_n)_\ast E_n \cong \map (\GG_n,(E_n)_\ast).
\end{equation}
and we have an Adams-Novikov Spectral
Sequence
\begin{equation}\label{ANSS}
H^s(\GG_n,(E_n)_\ast X) \Longrightarrow \pi_{t-s}L_{K(n)}X.
\end{equation}
The $E_2$-term is continuous group homology.

\begin{rem}\label{morava-st}
We give some more detail about the Morava stabilizer group $\GG_n$.
Write $\W$ for the Witt vectors $W(\f)$ and let $\sigma$
be the lift of Frobenius $\phi$ on $\f$ to $\W$. Define $\W \langle S\rangle$
to be the non-commutative polynomial ring over $\W$ with
$wS=S\sigma(w)$ for $w \in \W$. Then define
$$
\cO_n=\W \langle S\rangle /(S^n=p).
$$ 
There is an isomorphism  $\cO_n \cong \End(\Ga_n)$. From this it follows that
$\s_n \cong \cO_n^{\times}$.

Note the canonical right action of $\s_n$ on $\cO_n$ determines a homomorphism
$$
\s_n \longr \Aut(\cO_n) \cong \Gl_n(\W).
$$
It is easy to check that the image of the composition
$$
\xymatrix{
\s_n \rto & \Gl_n(\W) \rto^-{\mathrm{det}} &\W^\times
}
$$
is exactly $\ZZ_p^\times \subseteq \W^\times$, so that
we get an induced homomorphism
$$
\GG_n \longr  \Gal(\f/\FF_p) \times \ZZ_p^\times.
$$
The element $1+p \in \ZZ_p^\times$ defines an isomorphism 
$\ZZ_p \cong \ZZ_p^\times/\mu$, where $\mu$ is 
the maximal finite subgroup. Using this identification we get
the {\it reduced norm}.
\begin{equation}\label{red-norm}
\xymatrix{
N: \GG_n  \rto & \ZZ_p.
}
\end{equation}
Define $\GG_n^1$ to be the kernel of this homomorphism.

The inclusion $a \mapsto a + 0S \in \SS_n$ defines an injection
$\ZZ_p^\times \to \GG_n$ onto the center of $\GG_n$. We use
$1+p$ to identify  $\ZZ_p$ with the subgroup of $\ZZ_p^\times$ of elements
congruent to $1$ modulo $p$. Then the
composition
$$
\ZZ_p \to \GG_n \to \ZZ_p
$$
of the inclusion followed by the reduced norm of Equation (\ref{red-norm})
is multiplication by $n$. Hence if $n$ is prime to $p$, we have a splitting
isomorphism
\begin{equation}\label{split-gn1}
\xymatrix{
\GG_n^1 \times \ZZ_p \rto^-\cong &\GG_n.
}
\end{equation}
\end{rem}

\begin{rem}\label{action-of-center}
The action of the center on $(E_n)_\ast$ is
quite simple. The inclusion $\ZZ_p \to \End(\f,\Ga_n)$ is
the completion of the morphism $\ZZ \to \End(\f,\Ga_n)$ sending
$q$ to the group homomorphism $q:\Ga_n \to \Ga_n$. This has
a canonical lift to an action of the universal deformation and,
hence, $q \in \ZZ_p^\times$ acts trivially on $(E_n)_0$ and by
$$
q_\ast u = qu
$$
in degree $-2$.
\end{rem}

\subsection{Picard groups}\label{alg-pic-sec} The completed tensor product over
$(E_n)_\ast $ endows the category of twisted $\GG_n$-modules with
a symmetric monoidal structure whose unit is $(E_n)_\ast S^0$. We begin our analysis
of the Picard group $\Pic_n$ of $\cK_n$ with the following result from
\cite{HMS}.

\begin{thm}\label{pic-car} Let $X$ be a $K(n)$-local spectrum. 
Then the following conditions are equivalent:
\begin{enumerate}

\item[a)] $X$ is invertible in $\cK_n$.

\item[b)] $(E_n)_*X$ is a free $(E_n)_*$-module of rank $1$.

\item[c)]  $(E_n)_*X$ is invertible with respect to the tensor product in the category of
twisted $\GG_n$-modules.
\end{enumerate}
\end{thm} 

Write $\Picalg$ for the Picard group of the category
of twisted $\GG_n$-modules. Then
Theorem \ref{pic-car} implies that the assignment $X\mapsto (E_n)_*X$
defines a group homomorphism
$$
\varepsilon_n: \Pic_n\to \Picalg \ .
$$
The group $\Picalg$ contains the index $2$-subgroup 
$\PicalgO$ of invertible twisted $\GG_n$-modules
which are concentrated in even degrees. Likewise $\Pic_n$ contains 
a subgroup $Pic_n^0$ of index $2$ and $\epsilon$ restricts 
to a homomorphism 
$$
\varepsilon_n^0: \Pic_n^0\to (\Pic_n)_{\alg}^0
$$
If $M_\ast \in \PicalgO$, then $M_0$ is an invertible
$(E_n)_0$-module with an action of $\GG_n$. Standard
considerations now imply the following
result.

\begin{prop} There is a canonical isomorphism 
$$
(\Pic_n)_{\alg}^0\cong H^1(\GG_n,(E_n)_0^{\times})\ . 
$$
\end{prop}

\begin{rem}\label{basic-elements} There are two 
basic elements of $H^1(\GG_n,(E_n)_0^{\times})$ which
play a particularly important role in all calculations; for example,
in case $n=2$ they will be topological generators
of $H^1(\GG_n,(E_n)_0^{\times})$. Recall that $1$-cocycles
in group cohomology can be represented by crossed 
homomorphisms; that is, functions
$$
\phi:\GG_n \longr (E_n)^\times
$$
so that $\phi(gh) = [g\phi(h)]\phi(h)$. This formula is multiplicative
because the group operation on $(E_n)^\times$ is multiplication.

1.) The image of the the $K(n)$-local $2$ sphere  $L_{K(n)}S^2$ 
under the homomorphism $\varepsilon_n$ determines a crossed
homomorphism
$$
\eta:\GG_n\longr (E_n)_0^{\times}
$$
given by the formula
$$
g_\ast(u) = \eta(g)u
$$
where  $u \in (E_n)_{-2} = (E_n)_0S^2$
is the canonical generator.

2.) The second element, which we denote as  $(E_n)_0\langle\det\rangle$,
is the (genuine) homomorphism given as the composition of
the determinant map and the canonical inclusion
$$
\xymatrix{
\GG_n\rto^-\det & \Z_p^{\times}\rto^-\subseteq & (E_n)_0^{\times} \ .
}
$$
\end{rem}

\begin{rem}\label{det-realized}
The element $(E_n)_0\langle\det\rangle$ of $(\Pic_n)_{\alg}^0$
has a canonical topological realization, which we now define. We
assume $p > 2$.

Let $\SGG_n$ be the kernel of the determinant $\det:\GG_n \to \ZZ_p^\times$
and let $\mu \subseteq \ZZ_p^\times$ be the subgroup of $(p-1)$st roots of unity.
Then $\mu$ acts on $(E_n)^{h\SGG_n}$. Let $(E_n)^{h\SGG_n}_\chi$
be the wedge summand of $(E_n)^{h\SGG_n}$ defined by the
character $\chi:\mu\to \ZZ_p^\times$ with $\chi(g) = g^{-1}$. 
Let  $\psi^{p+1} = (p+1) + 0S \in \GG_n$ be our chosen generator of the  subgroup
of the center given by  elements congruent to $1$ modulo $p$; see Remark \ref{morava-st}.
Then $\psi^{p+1}$
maps to a topological generator of $\GG_n/\SGG_n$. 
Define $\Zdet$ by  the fiber sequence using the difference
\begin{equation}\label{SdetDefine}
\xymatrix{
\Zdet \ar[rr] && (E_n)^{h\SGG_n}_\chi \ar[rrr]^-{\psi^{p+1} - \det(\psi^{p+1})} &&
& (E_n)^{h\SGG_n}_\chi\ .
}
\end{equation}
Then we assert that $(E_n)_0\Zdet = (E_n)_0\langle\det\rangle$
as twisted $\GG_n$-module.

From \cite{GHMR} \S 2, we have an isomorphism
of Morava modules
\begin{equation}\label{ghmrS2}
(E_n)_0E_n^{h\SGG_n} \cong \map (\GG_n/\SGG_n,(E_n)_0)
\cong  \map (\ZZ_p^\times,(E_n)_0)\ .
\end{equation}
where the twisted $\GG_n$-action on the set of continuous maps is given by
$(g\phi)(x) = g\phi(g^{-1}x)$. It follows that
$$
(E_n)_0(E_n)^{h\SGG_n}_\chi\cong
\map(\ZZ_p,(E_n)_0\otimes \chi).
$$
The calculation
of $(E_n)_0\Zdet$ now follows from the long exact sequence in $(E_n)_\ast$ of the
fibration (\ref{SdetDefine}).
\end{rem}
 
\begin{defn}\label{exotic-pc} Define the group $\kappa_n$ of {\bf
exotic elements} of $\Pic_n$ to be the kernel of
$$
\varepsilon_n^0: \Pic_n^0\to (\Pic_n)_{\alg}^0 \ .
$$
\end{defn}
 
The group $\kappa_n$ measures the difference between the
homotopy theoretic and algebraic definition of the Picard group
in the $K(n)$-local category. The first result is negative: the following
Proposition says that $\kappa_n=0$ if $n$ is large with respect to $p$.
See \cite{Picard}.

\begin{prop}\label{vanish-kappa} Suppose $n$ is not divisible by 
$p-1$ and $n^2\leq 2p-2$. Then
$$
\kappa_n=0.
$$
\end{prop}

\begin{rem}\label{pic-at-1} The case $n=1$ was studied
for all primes in \cite{HMS}. Here $E_1$ can be chosen to be the 
$p$-complete $K$-theory and the group $\GG_1 \cong \ZZ_p^\times$ acts through
Adams operations. This group acts trivially on $(E_1)_0 = K_0= \ZZ_p$. If $p > 2$, then
$$
(\Pic_1)_{\alg}^0 \cong H^1(\ZZ_p^\times,\ZZ_p^\times ) \cong
 \ZZ_p \times \ZZ/(p-1)  \cong \ZZ_p^\times.
$$
and $\Pic_1 \cong (\Pic_1)_{\alg} \cong \ZZ_p \times \ZZ/2(p-1)$. Explicit
topological generators and other elements are given in \cite{HMS}.

If $p=2$, then $\GG_1=\ZZ_2^\times\cong \ZZ_2 \times \ZZ/2$
contains $2$-torsion. As a result, 
$$
(\Pic_1)_{\alg}^0 \cong H^1(\ZZ_2^\times,\ZZ_2^\times ) \cong
\ZZ_2^\times \times \ZZ/2. 
$$
The evaluation map
\begin{equation}\label{pic-ev-2}
\Pic^0_1 \to  (\Pic_1)^0_{\alg} \cong \ZZ_2^\times \times \ZZ/2
\end{equation}
sends $X$ to $(a_3,a_{-1})$ where $\psi^3(x) = a_3 x$ and $\psi^{-1}(x) = a_{-1} x$
for some generator $x \in K_0X$. For example, $S^2$ maps to $(3,-1)$. 
Notice that $(\Pic_1)^0_{\alg}$ has two generators of order $2$, with invariants
$(1,-1)$ and $(-1,1)$ respectively. We have
$$
(\Pic_1)_{\alg} \cong \ZZ_2 \times \ZZ/2 \times \ZZ/2
$$
generated by these elements and $K_\ast S^1$.

The map $\Pic_1 \to  (\Pic_1)_{\alg}$ is
a split surjection, but now $\kappa_1 \cong \ZZ/2$. An
explicit generator for $\kappa_1$ can be given by $L_{K(1)}DQ$
where $DQ$ is the dual of the ``question mark complex''. If we
extend $\eta:S^1 \to S^0$ to a map from the Moore spectrum
$$
\bar{\eta}:\Sigma M_2 \to S^0
$$
then $DQ$ is the cone of $\bar{\eta}$. A classical calculation shows
$K_\ast DQ \cong K_\ast S^0$ as graded modules over the Adams
operations; however, $KO \wedge DQ \simeq \Sigma^4 KO$.
\end{rem}

We now begin to consider the case  $n=2$ and $p > 2$.
The case $p=2$ is considerably harder and not discussed here at
all. The following  algebraic results are due to Hopkins \cite{hopu}
if $p > 3$ and to Karamanov if $p=3$ \cite{kar,kar1}. 

\begin{thm}\label{pnot2}  Let $n=2$ and $p>2$. Then 
\begin{enumerate}

\item $\PicalgO \cong H^1(\GG_n,(E_n)_0^{\times}) \cong
\Z_p^2\times \Z/(p^2-1)$ topologically generated by $(E_2)_0 S^2$
and $(E_2)_0\langle\det\rangle$;

\item $\Picalg \cong \Z_p^2\times \Z/2(p^2-1)$;

\item Both $\varepsilon_n$ and $\varepsilon_n^0$ are surjective,
and if $p>3$ they are both isomorphisms.
\end{enumerate}
\end{thm} 

The difficult part is the group cohomology calculation. Part (2) follows
from the observation, also in \cite{HMS}, that the extension
$$
0 \to \PicalgO \to \Picalg \to \ZZ/2 \to 0
$$
cannot be split. Part (3) follows from Proposition \ref{vanish-kappa}
and by noting, as in \cite{kar1}, that the elements of Remark
\ref{basic-elements}  generate the cohomology  group.

\subsection{Resolutions of $\GG_2$-modules}

The group $\GG_n$ is a $p$-adic analytic group in the sense of 
Lazard \cite{Laz} and such groups are of finite mod $p$ cohomological dimension 
unless they contain elements of order $p$.  We saw above in
Remark \ref{pic-at-1} that the element of order $2$ in $\GG_1$ at
$2$ created extra elements in $\Pic_1$. One
reason that this paper is interesting is that if $p=3$, then $\GG_2$ 
contains elements of order $3$.

{\bf From now on we will fix $p=3$ and work at $n=2$.}

An explicit element of order $3$ in $\GG_2$ is given by
$$
a=-\frac{1}{2}(1+\omega S)
$$ 
where $\omega$ is a fixed chosen primitive $8$-th root of unity in 
$\W=W(\FF_p)$. This element defines an inclusion $C_3 \to \GG_2$
of the cyclic group of order $3$ and, by \cite{HennDuke} Theorem 1.9,  the induced map
$$
H^\ast(\GG_2,\FF_3) \longr H^\ast(C_3,\FF_3)
$$
is surjective. Thus $\GG_2$ (at $p=3$) does not have finite
cohomological dimension. The further study of the cohomology
of $\GG_2$ (see Theorem \ref{alg-res} below) uses two subgroups.
Write $\langle - \rangle$
for the subgroup generated by a list of elements.

\begin{defn}\label{key-finite} Define two finite subgroups of 
$\GG_2$ as follows. Let $\phi \in \Gal(\FF_9/\FF_3)$ be the
Frobenius.
\begin{enumerate} 

\item $G_{24}=\langle a,\omega^2,\omega\phi\rangle$.  
Note that $\omega^2$  acts non-trivially on $C_3 = \langle a \rangle$
and $\omega\phi$ acts trivially on $C_3 = \langle a \rangle$; hence,
$G_{24} \cong C_3 \rtimes Q_8$ where $Q_8$ is the quaternion group of
order $8$.

\item  $SD_{16}=\langle\omega, \phi\rangle $. 
This group is isomorphic to the semidihedral group of order $16$. 
\end{enumerate} 
\end{defn}

Recall the splitting
$$
\GG_2 \cong \GG_2^1 \times \ZZ_3
$$
of Equation \ref{split-gn1}. Here $\GG_2^1$ is the kernel of the
reduced norm.  The finite subgroups of $\GG_2$ are automatically finite
subgroups of $\GG_2^1$. 

We now give a resolution of the trivial $\GG_2$-module $\ZZ_3$.
If $X = \lim_{\alpha}\ X_\alpha$ is a profinite set, let 
$$
\ZZ_p[[X]] = \lim_{i,\alpha}\ \ZZ/p^i[X_\alpha].
$$
Let $\chi$ be the character of $SD_{16}$ with
$\chi(\omega) = \chi(\phi) = -1$.

\begin{thm}{\cite{GHMR}}\label{alg-res}  There is an exact complex of 
$\Z_3[[\GG_2^1]]$-modules 
of the following form 
$$
0\to C_3\to C_2\to C_1\to C_0\to \Z_3 
$$ 
with $C_0=C_3\cong \Z_3[[\GG_2^1/G_{24}]]$ and 
$C_1=C_2\cong \Z_3[[\GG_2^1]]\otimes_{\Z_3[SD_{16}]}\Z_3(\chi)$.
\end{thm}

\begin{rem}\label{2from1} We can use this resolution to obtain a resolution
of the trivial $\GG_2$-module $\ZZ_3$. If $C_s$ is the $\GG_2^1$-module of
Theorem \ref{alg-res},  define $D_s$ to be the $\GG_2$-module obtained from $C_s$ by
inducing up from $\GG_2^1$. Thus, for example,
$$
D_0 \cong \ZZ_3[[\GG_2/G_{24}]].
$$
Using the isomorphism  $\GG_2^1 \times \ZZ_3 \cong \GG_2$ of (\ref{split-gn1}),
$D_s$ can be obtained by taking the completed tensor 
product of $C_s$ with $\ZZ_3[[\ZZ_3]]$. There is a larger resolutions by $\GG_2$-modules
\begin{equation}\label{alg-res-2}
0\to D_3\to D_3 \oplus D_2\to D_2 \oplus D_1\to D_1 \oplus D_0\to 
D_0 \to  \Z_3 
\end{equation}
To see this, let  $\psi^{p+1}=\psi^4$ be the chosen topological generator for
$\ZZ_3$. Then we get a very short resolution
\begin{equation}\label{alg-short-res}
\xymatrix@C=35pt{
0 \rto &\ZZ_3[[\ZZ_3]] \rto^-{\psi^{4} - 1} & \ZZ_3[[\ZZ_3]] \rto & \ZZ_3 \longr 0.
}
\end{equation}
If we write $P_\bullet$ for the complex
$$
P_\bullet \defeq
\{ \xymatrix@C=35pt{ \ZZ_3[[\ZZ_3]] \rto^-{\psi^{4} - 1} & \ZZ_3[[\ZZ_3]] }\}
$$
then the resolution $D_\bullet$ is the completion of the tensor product
$C_\bullet \otimes P_\bullet$,
where $C_\bullet$ is the resolution of Theorem \ref{alg-res}.
\end{rem}

We recall that a continuous $\ZZ_3[[\GG_2]]$-module $M$ is 
{\it profinite} if there is an isomorphism 
$M\cong\lim_{\a} M_\alpha$ 
where each $M_\alpha$ is a finite $\ZZ_3[[\GG_2]]$ module.

\begin{cor}\label{SS} Let $M$ be a profinite $\Z_3[[\GG_2^1]]$-module. 
Then there is a first quadrant cohomology spectral sequence 
$$
E_1^{s,t}(M)\cong \Ext_{\Z_3[[\GG_2^1]]}^s(C_t,M)\Longrightarrow 
H^{s+t}(\GG_2^1,M)
$$ 
with 
$$
E_1^{s,0}(M)=E_1^{s,3}(M)\cong H^s(G_{24},M)
$$ 
and 
$$
E_1^{s,1}(M)=E_1^{s,2}(M)\cong 
\begin{cases} 
\Hom_{\Z_3[SD_{16}]}(\Z_3(\chi),M) &  s=0  \\ 
0 & s>0 \ . 
\end{cases} 
$$ 
\end{cor}

\begin{rem}\label{ext-to-big} If $M$ is a profinite $\ZZ_3[[\GG_2]]$-module,
there is a similar spectral
sequence for computing $H^\ast(\GG_2,M)$ using the second
resolution (\ref{alg-res-2}), although the $E_1$ term
is slightly more complicated. We do record that
for $s > 0$, $E_1^{s,t} =0$ if $t=2$, and if $t=0,1,3$ or $4$,
then
$$
E_1^{s,t}(M) = H^s(G_{24},M).
$$
Furthermore, using Remark \ref{2from1}, we have
$$
d_1=\psi^{4}-1:E_1^{s,t}(M) \longr E_1^{s,t+1}(M)
$$
for $s > 0$ and $t=0,3$. 
\end{rem}

\begin{rem}\label{GHMR-input} We have considerable input for these spectral sequences.
For example,
for the spectral sequence of Corollary \ref{SS}
the terms $E_1^{\ast,1}$ and $E_1^{\ast,2}$ contribute
only to $H^1(\GG_2^1,M)$ and $H^2(\GG^1_2,M)$. These terms also can
be rewritten if $M = (E_2)_nX$ for some spectrum $X$. The representation
$\ZZ_3(\chi)$ is self-dual and $\ZZ_3(\chi) \otimes (E_2)_0 \cong (E_2)_0S^8
\cong (E_2)_{-8}$; hence,
$$
\Hom_{\Z_3[SD_{16}]}(\Z_3(\chi),(E_2)_\ast X) \cong H^0(SD_{16},(E_2)_\ast \Sigma^8X).
$$
Since $SD_{16}$ is a $2$-group, there are no higher cohomology groups.
In particular, if $X=S^0$, we have
$$
H^0(SD_{16},(E_2)_\ast) \cong \pi_\ast \Sigma^8E_2^{hSD_{16}}
$$
and we know from  \cite{GHMR} \S 3 that there is an isomorphism
$$
\ZZ_3[[y]][v_1,v_2^{\pm 1}]/(v_2y=v_1^4) \cong \pi_\ast \Sigma^8E_2^{hSD_{16}}.
$$
\end{rem}

For the terms where the group $G_{24}$ appears, we will use Theorem \ref{coh-G24}
below. There are invariant elements $c_4$, $c_6$ and $\Delta$ in
$H^0(G_{24},(E_2)_\ast)$ of internal degrees $8$, $12$ and $24$
respectively. The element $\Delta$ is invertible and there is a relation
\footnote{In \cite{GHMR} we wrote this relation as $c_4^3 - c_6^2
= 3^3\Delta$. However, we can replace $c_4$ and $c_6$ by 
$c_4/2^2$ and $c_6/2^3$ respectively to get the indicated relation,
which has the aesthetic value of coinciding with the standard relation
among modular forms. The connection can be made using the 
formal group of a supersingular elliptic curve. See \cite{GS}.}
$$
c_4^3 - c_6^2 = (12)^3\Delta\ .
$$
Define $j=c_4^3/\Delta$ and let
$M_\ast$ be the graded ring
$$
M_\ast = \ZZ_3[[j]][c_4,c_6,\Delta^{\pm 1}]/(c_4^3 - c_6^2 = (12)^3\Delta,
\Delta j = c_4^3).
$$
Furthermore, there are also elements $\alpha \in H^1(G_{24},(E_2)_4)$ and
$\beta \in H^2(\G_{24},(E_2)_{12})$ and there are relations
\begin{eqnarray}\label{mod-rels}
3\alpha = 3\beta = \alpha^2 &=0 \nonumber\\
c_4 \alpha = c_4\beta &=0\\
c_6 \alpha = c_6 \beta &=0 \nonumber
\end{eqnarray}

The fixed point spectral sequence is also computed in \S 3 in
\cite{GHMR}; the results will be useful in section 4 and 5. 

\begin{thm}\label{coh-G24} Let $R \subseteq M_\ast[\alpha,\beta]$
be the ideal generated by the relations of  (\ref{mod-rels}).
Then the induced map
$$
M_\ast[\alpha,\beta]/R \longr H^\ast(G_{24},(E_2)_\ast)
$$
is an isomorphism.
\end{thm}

\begin{thm}\label{hom-G24} 1.) In the fixed-point spectral
sequence
$$
H^s(G_{24},\Et_t) \Longrightarrow \pi_{t-s}E_2^{hG_{24}}
$$
all differentials are determined by
$$
d_5(\Delta) = \pm \alpha\beta^2
$$
and
$$
d_9(\alpha\Delta^2) = \pm \beta^5.
$$

2.)The class $\Delta^3$ is a permanent cycle and extends to an
equivalence
$$
\Sigma^{72}E_2^{hG_{24}} \simeq E_2^{hG_{24}}.
$$

3.) The kernel of the Hurewicz map 
$$
\pi_n E_2^{hG_{24}} \to \Et_n E_2^{hG_{24}}, \quad 0 \leq n \leq 72
$$
is a $\ZZ/3$ module generated by the classes
$$
\beta^i,\ 1 \leq i \leq 4;\qquad \alpha\beta^i, i = 0,1
$$
and classes $x$ and $\beta x$ where $x \in \pi_{27}E_2^{hG_{24}}$
is the Toda bracket $\langle\alpha,\alpha,\beta^2\rangle$
detected by $\pm \alpha\Delta$. Furthermore 
$\alpha x = \pm \beta^3.$
\end{thm}

\section{Exotic $\Pic$ at $p=3$}  

The main theorem of the paper is that $\kappa_2$ is isomorphic
to $\ZZ/3 \times \ZZ/3$.
We will state a refined version of this result below in
Theorem \ref{tau-iso} once we have assembled
the necessary preliminaries. This refined version will be
proved in the section 5.

Recall that we have an Adams-Novikov Spectral
Sequence
$$
E_2^{s,t}=H^s(\GG_2,(E_2)_t)\Rightarrow \pi_{t-s}(L_{K(2)}S^0).
$$

A key algebraic result is the following.

\begin{prop}\label{ANSS-dets} If $t$ is not divisible by $4 = 2(p-1)$, then
$$
H^s(\GG_2,(E_2)_t) = 0.
$$
Furthermore, there is
a splittable short exact sequence
\begin{equation}\label{ss1}
0 \to H^1(G_{24},\Et_4) \to H^5(\GG_2,\Et_4) \to H^5(G_{24},\Et_4)
\to 0\ .
\end{equation}
A choice of splitting yields an isomorphism
$$
H^5(\GG_2,(E_2)_4) \cong \Z/3\times\Z/3.
$$
\end{prop}

\begin{proof} The first statement is simply the standard sparseness
result for the Adams-Novikov Spectral Sequence. For a proof
in this context, see \cite{GHM}, Proposition 4.1.

For the second statement we
calculate using the spectral sequence of Remark \ref{ext-to-big},
which is a variant of the spectral sequence of Corollary 
\ref{SS}. From this spectral sequence and Theorem \ref{coh-G24}
we see that there is a short exact sequence
$$
0 \to E_2^{1,4} \to H^5(\GG_2,\Et_4) \to E_2^{5,0} \to 0
$$
where $E_2^{1,4}$ is cokernel of
$$
d_1: E_1^{1,3} = H^1(G_{24},\Et_4) \to H^1(G_{24},\Et_4) = E_1^{1,4}
$$
and $E_2^{5,0}$ is the kernel of
$$
d_1: E_1^{5,0} = H^5(G_{24},\Et_4) \to H^5(G_{24},\Et_4) = E_1^{5,1}.
$$
This differential is completely determined by the action of the center
of $\GG_2$; indeed, in both cases
$$
d_1 = (\psi^{4} -1)_\ast
$$
where $\psi^{4}$ is a generator for the central $\ZZ_3$.
It then follows from Remark \ref{action-of-center} that $d_1=0$. 

We now use Theorem \ref{coh-G24} to note that
$$
H^5(G_{24},(E_2)_4) \cong \Z/3
$$
generated by $\a\b^2\Delta^{-1}$ and
$$
H^1(G_{24},(E_2)_4) \cong \Z/3
$$
generated by $\alpha$. 

It remains to show that the exact sequence (\ref{ss1}) is split. For this
we compare it with the spectral sequence for the module
$M=(E_2/{(3,u_1)})_4=(\FF_9[u^{\pm 1}])_4$. If the
sequence of (\ref{ss1}) does not split, then the class $\alpha$
would map to zero under the induced map
$$
H^5(\GG_2,(E_2)_4) \to H^5(\GG_2,M).
$$
We will show that, in fact, $\alpha$ does not map to zero. The short exact
sequence of (\ref{ss1}) maps to the analogous short exact sequence
\begin{equation}\label{ss2}
0 \to H^1(G_{24},M) \to H^5(\GG_2,M) \to H^5(G_{24},M)
\to 0\ .
\end{equation}
Thus it is sufficient to argue that the reduction induces an isomorphism
$$
H^1(G_{24},(E_2)_4) \cong H^1(G_{24},M)
$$
and this follows immediately from Theorem \ref{coh-G24}.
\end{proof}

\begin{rem}\label{expl-gen} We can chose
one generator of $H^5(\GG_2,(E_2)_4)$
which restricts to the generator $\a\b^2\Delta^{-1}$
of $H^5(G_{24},(E_2)_4)$. Despite the fact that it is not unique,
we still call this generator $\a\b^2\Delta^{-1}$.
The other generates the kernel of
$$
H^5(\GG_2,E_2)_4)\to H^5(\GG_2^1,(E_2)_4).
$$
In the notation of \cite{GHMV1}, the image of this
element  in $H^5(\GG_2,(E_2)/(p,u_1)_4)$ is
$\zeta\a\b\a_{35}\Delta^{-2}$. We won't use this notation later in this
paper. 
\end{rem}

\begin{const}\label{tau-def} We next construct a homomorphism 
\begin{equation}\label{tau-eq}
\tau:\kappa_2\to H^5(\GG_2,(E_2)_4)
\end{equation}
from the group of exotic elements in $\Pic_2$.  The ideas here
can be found in \cite{HovSad} and \cite{shiPic}.

Let $Z\in \kappa_2$  and consider the
Adams Novikov Spectral Sequence
$$
H^s(\GG_2,(E_2)_tZ) \cong H^2(\GG_2,(E_2)_t Z)
\Longrightarrow \pi_{t-s}Z\ .
$$
A choice of isomorphism of twisted $\GG_2$-modules
$$
f: (E_2)_\ast \mathop{\longr}^{\cong} (E_2)_\ast Z
$$
defines a commutative diagram
$$
\xymatrix{
H^0(\GG_2,(E_2)_0) \rto^{\phi} \dto^\cong_{f_\ast} &H^5(\GG_2,(E_2)_4)
\dto^{f_\ast}_\cong\\
H^0(\GG_2,(E_2)_0Z) \rto_{d_5} & H^5(\GG_2,(E_2)_4Z)\ .
}
$$
We now set
$$
\tau(Z) = \phi(\iota) = f_\ast^{-1} d_5 f_\ast (\iota)
$$
where $\iota \in H^0(\GG_2,(E_2)_0)\cong \Z_3$ is the unit. To
check that this definition is independent of the chosen isomorphism
$f$, note that if $g$ is any other isomorphism, there is a unit
$a \in \Z_3^\times$ so that $g = af$. 
Finally, to check that $\tau$ is a homomorphism, we use the 
K\"unneth isomorphism
$$
(E_2)_\ast Z_1 \otimes_{(E_2)_\ast} (E_2)_\ast Z_2 \cong
(E_2)_\ast (Z_1\wedge Z_2)
$$
and the multiplicative structure of the Adams Novikov Spectral Sequence,
which gives a  Leibniz rule for differentials. 
\end{const}

\begin{rem}\label{pic-at-1-redux} In the case of $\Pic_1$ and $p=2$,
there is an analogous homomorphism
$$
\tau:\kappa_1 \to H^3(\ZZ_2^\times,(E_1)_2) \cong \ZZ/2
$$
which is an isomorphism. Compare Remark \ref{pic-at-1}.
\end{rem}

Our main technical theorem is the following result, proved in the
next section.

\begin{thm}\label{tau-iso} The homomorphism
$$
\tau:\kappa_2\to H^5(\GG_2,(E_2)_4)
$$
is an isomorphism.
\end{thm}

The following result will be needed later when calculating with our topological
resolutions. Recall from Construction \ref{tau-def} that a choice of isomorphism
$f:\Et_\ast \cong \Et_\ast Z$ defines a unit element 
$$
\iota_Z = f_\ast (\iota) \in H^0(\GG_2,E_0Z) \cong \ZZ_3
$$
unique up to  unit in $\ZZ_3$.

\begin{lem}\label{lem-inj-1} Let $Z \in \kappa_2$ be an exotic element in the Picard group.

1.) There is an integer
$k$, $0 \leq k \leq 2$, so that there is an equivalence of $E_2^{hG_{24}}$-modules
$$
\Sigma^{24k} E_2^{hG_{24}} \simeq E_2^{hG_{24}} \wedge Z.
$$

2.) There is an equivalence of $E_2^{hSD_{16}}$-modules 
$$
E_2^{hSD_{16}} \simeq E_2^{hSD_{16}} \wedge Z.
$$
\end{lem}

\begin{proof} For Part (1) we examine the fixed point spectral sequence reads
\begin{equation}\label{fpss}
H^s(G_{24},\Et_t Z) \Longrightarrow \pi_{t-s}(E_2\wedge Z)^{hG_{24}}.
\end{equation}
We note $(E_2\wedge Z)^{hG_{24}} \simeq E_2^{hG_{24}}\wedge Z$ because $Z$ is
a dualizable object in the $K(2)$-local category. This spectral sequence  (\ref{fpss})
is a module over the spectral sequence
$$
H^s(G_{24},\Et_t ) \Longrightarrow \pi_{t-s}E_2^{hG_{24}}.
$$
and the $E_2$ term of  (\ref{fpss}) is free of rank $1$ over $H^s(G_{24},\Et_t )$ on the class
$\iota_Z \in H^0(G_{24},\Et_0)$. There is an element $b \in \FF_3$ so that 
$$
d_5(\iota_Z ) = b\alpha\beta^2\Delta^{-1}\iota_Z\ .
$$
Then the differential $d_5(\Delta) = \pm \alpha\beta^2$ 
implies there is a $k$, $0 \leq k \leq 2$, so that 
$$
d_5(\Delta^k\iota_Z) = 0.
$$
For degree reasons, $\Delta^k\iota_Z$ is a permanent cycle. Comparing the
$E_\infty$-terms of the spectral sequence show that after extending the
resulting map $S^{24k} \to E_2^{hG_{24}} \wedge Z$ to a module map
$\Sigma^{24k} E_2^{hG_{24}} \to E_2^{hG_{24}} \wedge Z$ gives the needed equivalence.

Part (2) is similar, but easier, as $H^s(SD_{16},\Et_t Z) =H^s(SD_{16},\Et_t)=0$ for $s > 0$.
\end{proof}

\section{Topological resolutions and vanishing lines}

In order to prove Theorem \ref{tau-iso} and complete the calculation
of $\kappa_2$ we must use that the algebraic resolution of
Theorem \ref{alg-res} has a topological refinement. This is the 
main theorem of \cite{GHMR}, and we begin the section 
by recalling those results. This has other uses as well, and we will give  a proof
of the existence of a horizontal vanishing line for the Adams-Novikov Spectral
Sequence for any exotic element in the Picard group.

\subsection{Topological resolutions}

The isomorphism $(E_2)_\ast E_2\cong \map(\GG_2,\Et_\ast)$ 
of Equation (\ref{e-hom-e}) has the following refinement
for any closed subgroup of $K$ of $\GG_2$. As in \S 2 of \cite{GHMR},
there is an isomorphism, natural in $K$,
\begin{equation}\label{e-hom-e2}
(E_2)_\ast E_2^{hK} \cong \map(\GG_2/K,(E_2)_\ast).
\end{equation}
From this it follows that if we apply the functor
$$
\Hom_{\ZZ_3[[\GG_2]]}(-,(E_2)_\ast E_2) =
\Hom_{\ZZ_3[[\GG_2]]}(-,\map(\GG_2,(E_2)_\ast)
$$
to the first resolution of Theorem \ref{alg-res} induced up from
$\GG_2^1$ to $\GG_2$ we get a resolution
of $\Et_\ast E_2^{h\GG_2^1}$ by twisted $\GG_2$-modules
\begin{equation}\label{e-res-simp}
\begin{array}{ll}
\Et_\ast E_2^{h\GG_2^1} \to \Et_\ast E_2^{hG_{24}} \to&
\Et_\ast \Sigma^8E_2^{hSD_{16}}\\
\\
&\to \Et_\ast \Sigma^{40}E_2^{hSD_{16}} \to \Et_\ast \Sigma^{48}
E_2^{hG_{24}} \to 0\ .
\end{array}
\end{equation}
We have $\Sigma^8E_2^{hSD_{16}}$ because $C_1$ is twisted
by a character; also, $\Sigma^8E_2^{hSD_{16}} \simeq
\Sigma^{40}E_2^{hSD_{16}}$. See Remark \ref{GHMR-input}. However
$$
\Sigma^{48} E_2^{hG_{24}} \not\simeq E_2^{hG_{24}}
$$
even though
$$
\Et_\ast \Sigma^{48} E_2^{hG_{24}} \cong \Et_\ast E_2^{hG_{24}}.
$$
This suspension is crucial to 
 realization
of the resolution of (\ref{e-res}) of Theorem \ref{top-res}. See also
Remark \ref{alg-vs-top-period-1} for more on the role of suspensions.

For the sphere itself we use the second resolution 
of Theorem \ref{alg-res} to get a resolution of $\Et_\ast$ as
a twisted $\GG_2$-module:
\begin{equation}\label{e-res}
\begin{array}{ll}
\Et_\ast  \to &\Et_\ast E_2^{hG_{24}} \to
\Et_\ast \Sigma^8E_2^{hSD_{16}}\times \Et_\ast E_2^{hG_{24}}\\
\\
&\to \Et_\ast \Sigma^{40}E_2^{hSD_{16}}
\times \Et_\ast \Sigma^8E_2^{hSD_{16}}\\
\\
&\to \Et_\ast \Sigma^{48} E_2^{hG_{24}} 
\times \Et_\ast \Sigma^{40}E_2^{hSD_{16}} \to
\Et_\ast \Sigma^{48} E_2^{hG_{24}}  \to 0\ .
\end{array}
\end{equation}

The following theorem is the main result of \S 5 of \cite{GHMR}.

\begin{thm}\label{top-res} There is a sequence of maps between
spectra
\begin{align*}
L_{K(2)}S^0 \to E_2^{hG_{24}} \to
\Sigma^8E_2^{hSD_{16}} \times E_2^{hG_{24}}
& \to \Sigma^{40}E_2^{hSD_{16}}\times \Sigma^8E_2^{hSD_{16}}\\
\\
&\to \Sigma^{48} E_2^{hG_{24}} \times \Sigma^{40}E_2^{hSD_{16}} 
\to \Sigma^{48} E_2^{hG_{24}}
\end{align*}
realizing the resolution (\ref{e-res}) and
with the property that any two successive maps are null-homotopic
and all possible Toda brackets are zero modulo indeterminacy.
\end{thm}

There is an analogous topological resolution of $E_2^{h\GG_2^1}$ 
realizing the algebraic resolution of (\ref{e-res-simp}).

Let us write $F_s$ for the successive terms in the topological resolution
of Theorem \ref{top-res}. Thus $F_0 = E_2^{hG_{24}}$, 
$$
F_1 = \Sigma^8E_2^{hSD_{16}} \times E_2^{hG_{24}}\ ,
$$
on so on through $F_4$.  Then Theorem \ref{top-res} implies that
there is a tower of fibrations
\begin{equation}\label{key-tower}
\xymatrix@R=15pt{
L_{K(2)}S^0 \rto & X_3 \rto & X_2\rto & X_1 \rto &E^{hG_{24}}=F_0\ .\\
\Sigma^{-4}F_4\ar[u] & \Sigma^{-3}F_3\ar[u]&
\Sigma^{-2}F_2\ar[u] & \Sigma^{-1}F_1\ar[u]
}
\end{equation}
This tower yields a spectral sequence
$$
E_1^{s,t} = \pi_tF_s \Longrightarrow \pi_{t-s}L_{K(2)}S^0.
$$

\subsection{The vanishing line}

The $E_\infty$-term of the Adams-Novikov Spectral Sequence for exotic element $Z \in \kappa_2$
has a horizontal vanishing line at $s=13$. This is implicit in known results, especially work of 
Shimomura and his coauthors. For the sphere itself, see \cite{sh-w}. For a general
$Z$, the result can be deduced after a bit of work from \cite{IchSm}. We include here a
proof that uses our technology.

\begin{thm}\label{van-line} Let $Z \in \kappa_2$. Then the Adams-Novikov Spectral
Sequence
$$
E_2^{s,t}=H^s(\GG_2,(E_2)_t) \Longrightarrow \pi_{t-s}Z
$$
has a horizontal vanishing line
$$
E_{10}^{s,t}=0,\quad s \geq 13.
$$
\end{thm}

This implies, among other things, that the Adams-Novikov Spectral Sequence has only
$d_5$ and $d_9$. The proof will occupy the remainder of the section.

Let $\{X_q\}$ be the tower (\ref{key-tower}) with layers $F_q$ refining the topological resolution
of $L_{K(2)}S^0$. Let $C_0 = L_{K(2)}S^0$ and define spectra $C_q$ by the cofiber
sequence
$$
L_{K(2)}S^0 \to X_{q-1} \to \Sigma^{-q+1}C_q.
$$
Then there are cofiber sequences
$$
C_q \to F_q \to C_{q+1}
$$
which induce short exact sequences in $(E_2)_\ast$-homology. This can be summarized
in the diagram
\begin{equation}\label{exact-c}
\xymatrix@C=10pt@R=10pt{
L_{K(2)}S^0\ar[dr] && \ar@{-->}[ll]C_1\ar[dr]
&& \ar@{-->}[ll]C_2 \ar[dr]
&& \ar@{-->}[ll]C_3\ar[dr]
&& \ar@{-->}[ll]C_4\ar[dr]^\cong\\
&F_0 \ar[ur]&&F_1 \ar[ur]&&
F_2 \ar[ur]&&F_3 \ar[ur]
&&F_4
}
\end{equation}
where the dotted arrows are maps $C_q \to \Sigma C_{q-1}$.

For any spectrum
$X$ write $E_r^{s,t}X$ for the terms in the Adams-Novikov Spectral Sequence
$$
E_2^{s,t}X = H^s(\GG_2,(E_2)_tX) \Longrightarrow \pi_{t-s}L_{K(2)}X.
$$
Write $C_qX$ for $L_{K(2)}(C_q \wedge X)$ and
$F_qX$ for $L_{K(2)}(F_q \wedge X)$.

Using the cofiber sequences of (\ref{exact-c}) we 
then  obtain a spectral sequence
\begin{equation}\label{funny-ss}
E_1^{p,q,\ast}=H^p(\GG_2,(E_2)_\ast F_qX) \Longrightarrow H^{p+q}(\GG_2,(E_2)_\ast X)
\end{equation} 
from the exact couple
$$
\xymatrix@C=1pt@R=10pt{
E_2^{\ast,\ast}C_qX \ar[dr]
&& \ar@{-->}[ll]E_2^{\ast,\ast}C_{q+1}X\\
&E_2^{\ast,\ast}F_qX \ar[ur]}
$$
The dashed arrows raise cohomological degree by 1. This is isomorphic to the
spectral sequence obtained from the resolution of Remark \ref{2from1}; compare
Corollary \ref{SS}. 

\begin{lem}\label{hkm-coll} Let $Z \in \kappa_2$. The spectral sequence
\begin{equation}\label{funny-ss-z}
E_1^{p,q,\ast}=H^p(\GG_2,(E_2)_\ast F_qZ) \Longrightarrow H^{p+q}(\GG_2,(E_2)_\ast )
\end{equation} 
collapses at the first term for $p+q > 4$.  
\end{lem}

\begin{proof} It is sufficient to prove this for $Z = L_{K(2)}S^0$, as the spectral sequence 
(\ref{funny-ss-z})  doesn't depend on $Z \in \kappa_2$. 
Let $V(1)$ be the cofiber of the Adams map from
the Moore spectrum to itself. Then for all $q$ there is an injection
$$
H^p(\GG_2,(E_2)_\ast F_q) \to H^p(\GG_2,(E_2)_\ast F_q V(1))
$$
for $p > 1$. To see this, note that we need only show that
$$
H^p(G_{24},(E_2)_\ast) \longr H^p(G_{24},(E_2)_\ast/(3,u_1))
$$
is an injection for $p>1$ and this is a simple calculation. See, for example, Theorem A.2 of
\cite{HKM}.  

To complete the proof, we see note that the spectral sequence(\ref{funny-ss}) for $V(1)$
collapses at the second page, with only differentials $d_1$ on the $p=0$ line.
This follows from Theorem A.3 of \cite{HKM}.
From this we deduce that  in the spectral sequence for $S^0$, the longest potentially
non-zero differential is from $E_4^{0,3,\ast}$ to $E_4^{4,0,\ast}$. 
The result then follows.
\end{proof}

We will prove Theorem \ref{van-line} by comparing the Adams-Novikov Spectral Sequence
with another spectral sequence which agrees with the
Adams-Novikov Spectral Sequence in high degrees, but which is much more regular.

Let $\beta_1 \in \pi_{10}S^0$ be the usual element of order 3.
Define $\beta_1^{-1}S^0$ to be the homotopy colimit of
$$
\xymatrix{
S^0 \rto^{\beta_1} & S^{-10} \rto^{\beta_1} & S^{-20} \rto^{\beta_1} &\cdots
}
$$
As $\beta_1$ is nilpotent, the spectrum $\beta^{-1}S^0$ is contractible. 
If $X$ is any other spectrum, let 
$$
\beta_1^{-1} X = \beta_1^{-1}S^0 \wedge X.
$$
Since $\beta_1$ is detected by $\beta \in E_2^{2,12}S^0$, we then
obtain a localized spectral 
$$
\beta^{-1}E_2^{\ast,\ast}(X) \Longrightarrow \pi_\ast \beta^{-1}L_{K(2)}X = 0
$$
which, since colimits and limits do not always commute, does not obviously 
converge. In all our examples it will converge as we will have
$\beta^{-1}E_{10}^{\ast,\ast}(X)=0$.

\begin{exam}\label{tate-G24} As a warm-up, let's consider the case of
$\beta_1^{-1}E_2^{hG_{24}}$. Theorem \ref{coh-G24} gives an isomorphism
$$
\FF_3[\Delta^{\pm 1},\beta^{\pm 1}] \otimes \Lambda(\alpha)
\cong \beta^{-1}H^\ast(G_{24},(E_2)_\ast). 
$$
Then the differential $d_5(\Delta) =\pm  \alpha\beta^2$ (see Theorem \ref{hom-G24}) gives
$$
\beta^{-1}E_6^{\ast,\ast} = \FF_3[\Delta^{\pm 3},\beta^{\pm 1}] \otimes \Lambda(\alpha\Delta^2)
$$
and the differential $d_9(\alpha\Delta^2) = \pm \beta^5$ gives $\beta^{-1}E_{10}^{\ast,\ast}=0$.
The equivariantly minded reader will recognize $\beta_1^{-1}E_2^{hG_{24}}$ as the Tate
spectrum of the $G_{24}$-action on $E_2$.
\end{exam}

\begin{rem}\label{inv-sphere} We now examine the spectral sequence beginning with
$\beta^{-1}E_2^{\ast,\ast}S^0$  for the sphere
itself. First we combine Lemma \ref{hkm-coll} and Theorem \ref{coh-G24} to
deduce that $\beta^{-1}E_2^{\ast,\ast}S^0$ has a filtration with associated graded of the form
$$
\FF_3[\Delta^{\pm 1},\beta^{\pm 1}] \otimes \Lambda (\alpha, \zeta, e)
$$
The elements $\alpha$, $\beta$ and $\zeta$ come from the sphere and have
canonical lifts. We chose lifts of $\Delta$ and of $e$ and, at the risk of causing
a great deal of confusion, call the lifts by the same names as their residue classes.
In this way we obtain an isomorphism of filtered differential graded algebras 
$$
\FF_3[\Delta^{\pm 1},\beta^{\pm 1}] \otimes \Lambda (\alpha, \zeta, e)
\cong \beta^{-1}E_2^{\ast,\ast}Z.
$$
The differential is given by $d_5$. Recall that $F_3$ and $F_4$ have copies of $\Sigma^{48}
E_2^{hG_{24}}$; therefore,  applying Theorem \ref{hom-G24}, we deduce that
this differential satisfies
\begin{align*}
d_5(\Delta) &\equiv \pm \alpha\beta^2\\
d_5(e) &\equiv \pm \alpha\beta^2\Delta^{-1}e
\end{align*}
where $\equiv$ means ``modulo elements of higher filtration". All other generators
$x$ have $d_5(x) =0$, since they detect permanent cycles in $E_2^{\ast,\ast}S^0$.

We now have a spectral sequence for $\beta^{-1}E_6^{\ast,\ast}S^0$ which initial
term
$$
\FF_3[\Delta^{\pm 3},\beta^{\pm 1}]\{1,\alpha\Delta^2\}
\otimes \Lambda (\zeta, \Delta^2e).
$$
The differential $d_r$ of this auxiliary spectral sequence changes degree as follows:
$$
(p,q,t) \longmapsto (p+5-r,q+r,t+4).
$$
As a result, the spectral sequence collapses for degree reasons and hence
$\beta^{-1}E_6^{\ast,\ast}S^0$ is isomorphic to this algebra, with the choices above.
They only other differential remaining is $d_9$, which is determined by
$$
d_9 (\alpha\Delta^2) \equiv \pm \beta^5.
$$
The differential on all other generators is either zero outright or zero modulo
elements of higher filtration. We have $\beta^{-1}E_{10}^{\ast,\ast}S^0 = 0$.
\end{rem}

\begin{rem}\label{inv-gen} Now let $Z \in \kappa_2$ be any exotic element of the
Picard group. By similar reasoning to that of Remark \ref{inv-sphere}, we  deduce
that $\beta^{-1}E_2^{\ast,\ast}Z$ is
a filtered differential module over $\beta^{-1}E_2^{\ast,\ast}S^0$ which is free of rank
$1$ on a generator $\iota_Z$ of bidegree $(0,0)$. The differentials in
$\beta^{-1}E_2^{\ast,\ast}Z$  are determined by the differentials in $\beta^{-1}E_2^{\ast,\ast}S^0$
and the differential
$$
d_5(\iota_Z) \equiv \mp k\alpha\beta^2\Delta^{-1}\iota_Z
$$
where $k$, $0 \leq k \leq 2$, is the integer so that $E_2^{hG_{24}}\wedge Z =
\Sigma^{24k}E_2^{hG_{24}}$. 

Again arguing as in Remark \ref{inv-sphere}, we see that $\beta^{-1}E_6^{\ast,\ast}Z$ is free of
rank $1$ over $\beta^{-1}E_6^{\ast,\ast}S^0$ on $\Delta^k\iota_Z$; again we have
$\beta^{-1}E_{10}^{\ast,\ast}Z=0$.
\end{rem}

We now begin an analysis of the localization map of spectral sequences.
From Lemma \ref{hkm-coll}, we know
$$
E_2^{\ast,\ast}Z \to \beta^{-1}E_2^{\ast,\ast}Z \cong [\FF_3[\Delta^{\pm 1},\beta^{\pm 1}]
\otimes \Lambda (\alpha, \zeta, e)]\iota_Z
$$
is an isomorphism in homological degrees $s > 4$. To get some hold on the image define a 
sub-differential graded module
$$
A \defeq \FF_3[\Delta^{\pm 1},\beta]\{\beta^2,\alpha\beta^2,\beta^2\zeta,
\alpha\beta\zeta,
\beta e,\alpha e,\beta\zeta e, \alpha \zeta e\} \subseteq \beta^{-1}E_2^{\ast,\ast}S^0.
$$
This inclusion is also an isomorphism in degrees $s > 4$. Since the 
longest differential in the algebraic spectral sequence (\ref{funny-ss}) is
a $d_4$, the inclusion of $A \to \beta^{-1}E_2^{\ast,\ast}S^0$ factors through
$E_2^{\ast,\ast}S^0 \to  \beta^{-1}E_2^{\ast,\ast}S^0$. This factoring is unique in
degrees $s > 4$, which implies that the factoring automatically commutes with
$d_5$. Choose a factoring and identify $A$ with its image in $E_2^{\ast,\ast}S^0$.

Next observe that 
$A\iota_Z \subseteq E_2^{\ast,\ast}Z$ is an inclusion of a differential submodule
and an isomorphism in degree $s > 4$. From this we deduce that 
the map
$$
H^s(A\iota_Z,d_5) \to E_6^{s,\ast}Z
$$
is onto for $s > 4$ and an isomorphism for $s > 9$. This implies that the map
$$
E_6^{\ast,\ast}Z \to \beta^{-1}E_6^{\ast,\ast}Z
$$
is an isomorphism in degrees $s > 9$.

We now move on to the calculation of $d_9$. The inclusion
of $A$, as above, defines an inclusion
$$
B\defeq \FF_3[\Delta^{\pm 3},\beta]\{\beta^2,\alpha\beta^2\Delta^2,\beta^2\zeta,
\alpha\beta \zeta\Delta^2,
\beta\Delta^2e,\alpha\Delta^4 e,\beta\Delta^2\zeta e, \alpha\Delta^4
 \zeta e\} \subseteq H^\ast A \to  E_6^{\ast,\ast}S^0
$$
which maps injectively through to $\beta^{-1}E_6^{\ast,\ast}S^0$.
Then
$$
B\Delta^k\iota_Z  \to \beta^{-1}E_9^{s,\ast}Z
$$
is onto in degrees $s > 4$ and we have
that the inclusion $B\Delta^k\iota_Z  \to E_6^{\ast,\ast}Z$ is an isomorphism in degrees $s > 9$. 
Since  $B$ has no elements of degree $s=0$, $B$ is closed under $d_9$.
We then have 
$$
H^\ast (B\Delta^k\iota_Z ,d_9) \cong E_{10}^{\ast,\ast}Z
$$
is an isomorphism for degree $s > 10$. The $d_9$ differentials in $B$ are 
all induced from the formula $d_9(\alpha\Delta^2) = \pm \beta^5$; hence,
$H^\ast (B,d_9)=0$ in degrees $s > 12$ and we can deduce Theorem \ref{van-line}.

There are non-zero elements of degree $s=12$; for example, the elements
$$
e\zeta\Delta^2\Delta^{3k}\beta^4
$$
are of $s$ filtration 12 and non-zero in $E_{10}^{\ast,\ast}S^0$. Many of these detect permanent 
cycles in $\pi_\ast L_{K(2)}V(1)$; see \cite{GHMV1}.

\begin{cor}\label{d5saysit} Let $Z$ be an element of $\kappa_2$ so that $d_5(\iota_Z) = 0$.
Then $\iota_Z$ is a permanent cycle and $L_{K(2)}S^0 \simeq Z$.
\end{cor}

\begin{proof} By the vanishing result, we need only check that $d_9(\iota_Z)=0$.
In the spectral sequence of Remark \ref{ext-to-big} for computing $H^\ast(\GG_2,(E_2)_8)$
all groups in total degree $9$ vanish, hence $H^9(\GG_2,(E_2)_8) = 0$.
\end{proof}

\section{The decomposition of the group of exotic elements}

In order to prove Theorem \ref{tau-iso} and complete the calculation
of $\kappa_2$  we discuss which exotic
elements in $\kappa_2$ can be detected by $E_2^{hG_{24}}$.
This will leave a subgroup of order $3$ which cannot be seen
by $E_2^{hG_{24}}$; this subgroup is discussed in the last subsection.

Here is an outline of the arguments of this section.
Let $\kappa_2^1$ denote the set of weak
equivalence classes of $E_2^{hG_{24}}$-modules $X$ so that
$$
\Et_\ast X \cong \Et_\ast E_2^{hG_{24}}
$$
as twisted $\GG_2$-modules. Then $\kappa^1_2$ is a group
under smash product over $E_2^{hG_{24}}$. There is
a homomorphism
$$
\tau^1:\kappa_2^1 \longr H^5(G_{24},\Et_4) \cong \ZZ/3
$$
defined, as in Construction \ref{tau-def}, using $d_5$ in the Adams-Novikov
Spectral Sequence and the Schapiro isomorphism
$$
H^\ast (\GG_2, (E_2)_\ast X) \cong H^\ast(G_{24},(E_2)_\ast).
$$
We will see in Proposition
\ref{alg-vs-top-period-2} that  this map is a surjection. It is also an injection,
for if $d_5(\iota_X) = 0$, then $\iota_X$ is a permanent cycle
and the resulting homotopy class can be extended to
an equivalence $E_2^{hG_{24}} \simeq X$.

There is a map $\kappa_2 \to \kappa_2^1$ sending $Z$
to $E_2^{hG_{24}}\wedge Z$ and a commutative
diagram
$$
\xymatrix{
\kappa_2 \rto^-\tau \dto & H^5(\GG_2,\Et_4) \dto\\
\kappa_2^1 \rto^-\cong_-{\tau^1} & H^5(G_{24},\Et_4)
}
$$
where the map on group cohomology is the restriction.
Theorem \ref{exotic-eo} below shows that the map $\kappa_2 \to
\kappa_2^1$ is onto. We now let $\kappa_2^0$ be the kernel. From
Proposition \ref{ANSS-dets} we then get an induced map
$$
\tau^0:\kappa_2^0 \longr H^1(G_{24},\Et_4) \cong \ZZ/3
$$
and we must show that this map is an isomorphism. This is accomplished
in Propositions \ref{tau-inj} and \ref{tau-sur-1}.

\subsection{Exotic elements detected by $E_2^{hG_{24}}$} Let's begin
by revisiting the case $n=1$ and $p=2$. See Remark \ref{pic-at-1}.

\begin{exam}\label{back-to-one} At $p=2$ and $n=1$, there is very short
resolution
\begin{equation}\label{alg-res-k}
\xymatrix{
K_\ast \rto & K_\ast KO \rto^{\psi^3-1} & K_\ast KO \rto & 0
}
\end{equation}
of $K_\ast = K_\ast S^0$ as a continuous module over the Adams operations.
This can be realized by the fiber sequence
$$
\xymatrix{
L_{K(1)}S^0 \rto & KO \rto^{\psi^3-1} & KO.
}
$$
However, there is a another topological realization of the resolution
(\ref{alg-res-k}) using the fact that the Bott periodicity isomorphism
$$
\Sigma^4 K_\ast K \mathop{\longr}^{\cong} K_\ast K
$$
induces an isomorphism of $\GG_1$-modules
$$
\Sigma^4 K_\ast  KO \mathop{\longr}^{\cong} K_\ast KO
$$
which cannot be realized topologically. We then get a fiber sequence
$$
\xymatrix{
X \rto & \Sigma^4 KO \ar[rr]^{\Sigma^4 3^{-2}\psi^3-1}
&& \Sigma^4 KO
}
$$
defining an exotic element $X \in \kappa_1$. The map $X \to \Sigma^4KO$ extends to 
a weak equivalence $KO \wedge X \simeq \Sigma^4 KO$ of $KO$-modules. Since 
$\kappa_1 \cong \ZZ/2$, this implies $X \simeq L_{K(1)}DQ$, the non-zero
element. Note this construction produces only $L_{K(1)}DQ$ and not the finite complex $DQ$
itself.
\end{exam}

We now do something similar, but slightly more elaborate, at the prime $3$. 
Let $G \subseteq \GG_2$ be a finite
subgroup. The invariants $H^0(G,(E_2)_\ast)$ form a graded ring.
Suppose we can choose an invertible element $x$ in this ring
of degree $t$. Then $x$ defines an isomorphism of
$(E_2)_\ast[G]$-modules
$$
x:\Sigma^t (E_2)_\ast \mathop{\longr}^{\cong} (E_2)_\ast
$$
which extends to an isomorphism of twisted $\GG_2$-modules
$$
\alpha:\Sigma^t \map(\GG_2/G,(E_n)_\ast) \mathop{\longr}^{\cong}
\map(\GG_2/G,(E_2)_\ast)
$$
given by
$$
\alpha(\Sigma^t\phi)(g) = (gx)\phi(g).
$$
We are using the diagonal $\GG_2$-action; compare Equation \ref{diagonal}.
This isomorphism of Equation \ref{e-hom-e} induces an isomorphism
$$
(E_2)_\ast E_2^{hG} \cong \map(\GG_n/G,(E_2)_\ast)
$$
and so $x$ yields an algebraic isomorphism
$$
(E_2)_\ast \Sigma^t E_2^{hG} \mathop{\longr}^{\cong}
(E_2)_\ast E_2^{hG}
$$
which may not be realized topologically. If $x$ is a unit of minimal
positive degree, then we say that $E_2^{hG}$ has {\it algebraic periodicity}
$t$ and $x$ is a periodicity operator. This construction works,
of course, at all primes $p$ and all $n$. From \S 3 of \cite{GHMR}
we have the following result. See also Theorem \ref{coh-G24}.

\begin{lem}\label{alg-vs-top-period}  1.) The fixed point spectrum
$E_2^{hG_{24}}$ has algebraic
periodicity $24$ and
$\Delta \in H^0(G_{24},(E_2)_{24})$ is a periodicity operator.

2.) The fixed point spectrum $E_2^{hSD_{16}}$  has algebraic
periodicity $16$ and $v_2 = u^{-8}$ is a periodicity operator.
\end{lem}

\begin{rem}\label{alg-vs-top-period-1} The algebraic periodicity
on $E_2^{hSD_{16}}$ can be realized topologically, because
there can be no differentials in the fixed point spectral sequence.
However, the algebraic periodicity of $E_2^{hG_{24}}$ cannot
be realized topologically as $\Delta$ is not a permanent cycle.
(See Theorem \ref{hom-G24}.) Topologically this spectrum has periodicity
$72$ with periodicity operator $\Delta^3$.
\end{rem}

\begin{prop}\label{alg-vs-top-period-2} The homomorphism
$$
\tau^1:\kappa_2^1 \to H^5(G_{24},\Et_4)\cong \Z/3
$$
is an isomorphism.
\end{prop}

\begin{proof} We already noted at the beginning of the section
that the map is a monomorphism. Remark \ref{alg-vs-top-period-1}
implies that it is also an epimorphism.
\end{proof}

The next result now says that the map
$\kappa_2 \to \kappa_2^1$ sending $Z$ to $E_2^{hG_{24}}\wedge Z$
is surjective.

\begin{thm}\label{exotic-eo}There exists an exotic
element $P \in \kappa_2$ so that
$$
E_2^{hG_{24}} \wedge P \simeq \Sigma^{48}E_2^{hG_{24}}.
$$
\end{thm}

\begin{proof} We proceed as in Example \ref{back-to-one} and twist
the resolution of Theorem \ref{top-res} in order to produce
$P$. We use the isomorphisms
$$
\Et_\ast \Sigma^{48}E_2^{hF} \cong \Et_\ast E_2^{hF}
$$
and the resolution \ref{e-res} to produce a new resolution
$$
\begin{array}{ll}
\Et_\ast  \to &\Et_\ast \Sigma^{48}E_2^{hG_{24}} \to
\Et_\ast \Sigma^{56}E_2^{hSD_{16}}\times \Et_\ast 
\Sigma^{48}E_2^{hG_{24}}\\
\\
&\to \Et_\ast \Sigma^{88}E_2^{hSD_{16}}
\times \Et_\ast \Sigma^{56}E_2^{hSD_{16}}\\
\\
&\to \Et_\ast \Sigma^{96} E_2^{hG_{24}} 
\times \Et_\ast \Sigma^{88}E_2^{hSD_{16}} \to
\Et_\ast \Sigma^{96} E_2^{hG_{24}}  \to 0
\end{array}
$$
This has a topological realization. In fact, the arguments of
Theorem 5.5 of \cite{GHMR} go through verbatim to
produce the sequence
\begin{align*}
\Sigma^{48} E_2^{hG_{24}} \to
\Sigma^{56} E_2^{hSD_{16}} \times \Sigma^{48} E_2^{hG_{24}}
& \to \Sigma^{88}E_2^{hSD_{16}}\times \Sigma^{56}E_2^{hSD_{16}}\\
\\
&\to \Sigma^{96} E_2^{hG_{24}} \times \Sigma^{88}E_2^{hSD_{16}} 
\to \Sigma^{96} E_2^{hG_{24}} 
\end{align*}
with all the necessary Toda brackets zero modulo
indeterminacy. Notice that each of the spectra in this new
resolution are spectra of the old resolution, suspended $48$ times;
however, the maps are not simple suspensions of the old maps.
Then $P$ is the top of the resulting tower. The induced
map $P \to \Sigma^{48}E_2^{hG_{24}}$
from the top of the tower to the bottom extends a map
$$
E_2^{hG_{24}} \wedge P \to \Sigma^{48}E_2^{hG_{24}}
$$
of $E_2^{hG_{24}}$-modules and we'd like to see that this is an equivalence.
To see this, we consider the following commutative diagram
$$
\xymatrix{
(E_2)_\ast (E_2^{hG_{24}} \wedge P) \rto \dto &(E_2)_\ast \Sigma^{48}E_2^{hG_{24}}\dto\\
\map(\GG_2/G_{24},(E_2)_\ast X) \rto & \map(\GG_2/G_{24},\Sigma^{48}(E_2)_\ast ).
}
$$
The vertical maps are the isomorphisms of (\ref{e-hom-e2}) and the lower map is the
isomorphism induced by algebraic periodicity. It follows that the upper map is 
an isomorphism, as needed.
\end{proof}

\begin{rem}\label{uniqueness-of-P} The element $P$ is not uniquely
determined by the requirement that $E_2^{hG_{24}} \wedge P
\simeq E_2^{hG_{24}}$. However, we show in \cite{bcdual} that
$P$ is the unique element in $\kappa_2$ so that
\begin{enumerate}

\item $E_2^{hG_{24}} \wedge P \simeq \Sigma^{48} E_2^{hG_{24}}$, and

\item $P \wedge V(1) \simeq \Sigma^{48}L_{K(2)}V(1)$, where $V(1)$ is the
cofiber of the Adams map on the mod-$3$ Moore spectrum.
\end{enumerate}
\end{rem}

\subsection{The truly exotic elements}

Recall that $\kappa^0_2$ is the subgroup of $\kappa_2$ consisting
of the exotic elements in $Z\in \Pic(\cK_2)$ so that there is
an equivalence of $E_2^{hG_{24}}$-modules
$$
E_2^{hG_{24}} \wedge Z \simeq E_2^{hG_{24}}\ .
$$
We must now show that $\kappa_2^0 \cong H^1(G_{24},\Et_4)$.

We prove the injectivity statement of Theorem \ref{tau-iso}.

\begin{prop}\label{tau-inj} The homomorphism
$$
\tau^0:\kappa_2^0\to H^1(G_{24},(E_2)_4) \subseteq H^5(\GG_2,\Et_4)
$$
is injective.
\end{prop}

\begin{proof} This is a restatement of Proposition \ref{ANSS-dets} and Corollary
\ref{d5saysit}.
\end{proof}

It remains to show that $\tau:\kappa_2\to H^5(\GG_2,(E_2)_4)$
is surjective. We will prove this by constructing explicit elements
in $\kappa_2$.

\def\egone{{{E_2^{h\GG^1_2}}}}

To prepare for the argument, consider the tower realizing the
resolution of Equation (\ref{e-res-simp})
\begin{equation}\label{key-tower-1}
\xymatrix@R=15pt{
\egone \rto & X_2 \rto & X_1\rto  &E^{hG_{24}}=F^1_0\ .\\
\Sigma^{-3}F^1_3\ar[u]&
\Sigma^{-2}F^1_2\ar[u] & \Sigma^{-1}F^1_1\ar[u]
}
\end{equation}

\begin{lem}\label{egone} Let $Z \in \kappa_2^0$.
Then there is  a weak equivalence of $\egone$-modules
\begin{equation*}
\egone \simeq \egone \wedge Z.
\end{equation*}
\end{lem}

\begin{proof} There is no obstruction to lifting the map $S^0 \to E_2^{hG_{24}} \wedge Z$
detected by $\iota_Z$ up the tower obtained from (\ref{key-tower-1}) by applying $(-) \wedge Z$.
\end{proof}

There is a short fiber sequence
\begin{equation}\label{s-fiber-2}
\xymatrix{
L_{K(2)}S^0 \rto & \egone \rto^{\psi^4-1} & \egone
}
\end{equation}
where $\psi^4=\psi^{p+1}$ is a generator for the central $\Z_3\cong
\Z_3^{\times}/\{\pm 1\}$ acting on $\egone$. We apply
$(-)\wedge Z$ and the equivalence of Lemma \ref{egone} to obtain
a fiber sequence
$$
\xymatrix{
Z \rto & \egone \rto^{f_Z} & \egone.
}
$$
Let $\iota \in \pi_0\egone$ be the unit element.
Then $Z \simeq L_{K(2)}S^0$ if and only if 
$$
(f_Z)_\ast(\iota) = 0.
$$
An analysis
of the homotopy groups of $\egone$ using the tower (\ref{key-tower-1})
shows there is an exact sequence
$$
0 \to \ZZ/3 = \pi_{-45}E_2^{hG_{24}} \to \pi_0\egone
\to \pi_0E_2^{hG_{24}}
$$
and the $E_2$-Hurewicz homomorphism shows that $(f_Z)_\ast (\iota)$
must land in $\pi_{-45}E_2^{hG_{24}}$. Thus, to construct the non-trivial
exotic $Z$ needed to show $\tau^0$ is a surjection,  we need only produce enough self-maps
for $\egone$ to  realize this obstruction. This we do in the proof of the following result.

\begin{thm}\label{tau-sur-1} Let $\kappa^0_2$ be the subgroup
of $\kappa_2$ consisting of those elements with 
$E_2^{hG_{24}}\wedge Z \simeq E_2^{hG_{24}}$.
The homomorphism
\begin{align*}
\tau^0: \kappa_2^0 &\longr  H^1(G_{24},(E_2)_4) \subseteq H^5(\GG_2,(E_2)_4)
\end{align*}
is surjective. 
\end{thm}

\begin{proof} We begin with a calculation 
of $[E_2^{h\GG_2^1},E_2^{h\GG_2^1}]$.
Apply the functor $[E_2^{h\GG_2^1},-]$ to the tower
(\ref{key-tower-1}), then use Proposition 2.6 of \cite{GHMR} and
the calculations of Theorem \ref{hom-G24} to get an exact sequence
\begin{align*}
0\to  \pi_3(\Sigma^{48}E_2^{hG_{24}}[[\GG_2/\GG_2^1]])
\to &[E_2^{h\GG_2^1},E_2^{h\GG_2^1}]\\
&\to \pi_0(E_2^{hG_{24}})[[\GG_2/\GG_2^1]]\to
\pi_0(\Sigma^{8}E_2^{hSD_{16}})[[\GG_2/\GG_2^1]].
\end{align*}
The last map is induced by $d_1;E_2^{hG_{24}} \to \Sigma^8E_2^{hSD_{16}}$
at the beginning of the topological resolution realizing (\ref{e-res-simp}).
In Theorem 4.3 of \cite{GHM} we proved that the kernel of 
$(d_1)_\ast:\pi_0 E_2^{hG_{24}} \to \pi_0\Sigma^8E_2^{hSD_{16}}$
is $\ZZ_3$ generated by the multiplicative identity. Thus we have a
short exact sequence
\begin{equation}\label{pre-T}
0\to  \pi_3(\Sigma^{48}E_2^{hG_{24}}[[\GG_2/\GG_2^1]])
\to [E_2^{h\GG_2^1},E_2^{h\GG_2^1}]
\to \ZZ_3[[\GG_2/\GG_2^1]]\to 0\ .
\end{equation}

Since $\GG_2/\GG_2^1 \cong \ZZ_3$, there is an isomorphism of
complete local rings
$$
\ZZ_3[[T]] \mathop{\longr}^{\cong} \ZZ_3[[\GG_2/\GG_2^1]]
$$
sending $T$ to $\psi-1$. Again from Theorem \ref{hom-G24} we then
have
$$
\pi_3(\Sigma^{48}E_2^{hG_{24}}[[\GG_2/\GG_2^1]])\cong \Z/3[[T]],
$$
We also have
\begin{equation}\label{post-T}
0\to \Z/3\cong \pi_3(\Sigma^{48}E_2^{hG_{24}})\to
\pi_0(E_2^{h\GG_2^1}) \to  \ZZ_3\to 0  \ .
\end{equation}
The final $\ZZ_3$ is the kernel of the map $(d_1)_\ast:\pi_0 E_2^{hG_{24}} \to 
\pi_0\Sigma^8E_2^{hSD_{16}}$, as above.
We can obtain the sequence (\ref{post-T}) by direct calculation
or from Equation \ref{pre-T} by killing the action of $T$ with
the fibration sequence of Equation (\ref{s-fiber-2}). Combining
the exact sequences of Equations \ref{pre-T} and \ref{post-T} we
obtain a commutative diagram with the vertical maps induced
by the unit map $S^0 \to E_2^{h\GG^1_2}$.
\begin{equation}\label{all-T}
\xymatrix{
0 \rto& \ZZ/3[[T]] \rto \dto_{0=T} &  [E_2^{h\GG_2^1},E_2^{h\GG_2^1}] \rto\dto
& \Z_3[[T]] \rto \dto^{T=0} & 0\\
0 \rto& \ZZ/3\rto  &  \pi_0E_2^{h\GG_2^1} \rto
&\ZZ_3 \rto  & 0\ .
}
\end{equation}

Let $b \in \pi_3(\Sigma^{48}E_2^{hG_{24}}[[\GG_2/\GG_2^1]])
\cong \Z/3[[T]]$ and define spectra $Q_b$ by the fibration
sequences
\begin{equation}\label{defnQb}
\xymatrix{
Q_b \rto & E_2^{h\GG_2^1} \rto^{T + b} & E_2^{h\GG_2^1}.
}
\end{equation}
Recall $T = \psi-1$.
Since $\Et_\ast b = 0$, we have $\Et_\ast Q_b \cong \Et_\ast$
as twisted $\GG_2$-modules; hence $Q_b \in \kappa_2^0$.
Furthermore, from the diagram of Equation \ref{all-T}, we see that
the unit element $\pi_0E_2^{h\GG_2^1}$ factors through $Q_b$
if and only if $b$ is divisible by $T$. In particular if
$b \not\equiv 0$ modulo $T$, then 
$$
\tau(Q_b) \ne 0 \in \pi_{-45}E_2^{hG_{24}}.
$$
\end{proof}

\begin{rem}\label{constructable} We  constructed all the elements of
$\kappa_2$ in Theorem \ref{exotic-eo} and Theorem \ref{tau-sur-1}.
However, these constructions write the exotic elements as
homotopy inverse limits of diagrams built from the infinite spectra
$E_2^{hG}$, with $G$ finite.  What we have not done is  construct finite
CW spectra whose $K(n)$-localizations realize these elements.
This is in contrast to the case $n=1$ and $p=2$, where the exotic
element is the localization of an explicit three-cell complex.
\end{rem}

\section{The computation of $Pic(\cL_2)$} 

We begin with some generalities. Recall that $\Laa_n$ denotes the
category of  $L_n$-local spectra and $\Ka_n$ the category of
$K(n)$-local spectra.

\begin{lem}\label{hom-on-pic} Localization defines a homomorphism
$$
L_{K(n)}: \Pic(\cL_n) \longr \Pic_n
$$
which restricts to a homomorphism $\kappa(\cL_n) \to \kappa_n$.
\end{lem}

\begin{proof} If $X$ and $Y$ are spectra, then the natural map
$X \wedge Y \to L_{K(n)}X \wedge L_{K(n)}Y$ induces an
equivalence
$$
L_{K(n)}(X \wedge Y) \to  L_{K(n)}(L_{K(n)}X \wedge L_{K(n)}Y).
$$
The result follows from the definition of the Picard group.
\end{proof}

A first calculation is due to Hovey and Sadofsky \cite{HovSad}:

\begin{prop}\label{pic-l-n} Let $X \in \Pic(\Laa_n)$. Then there is an integer 
$k$ so that
$E(n)_\ast X \cong E(n)_\ast S^k$ as $E(n)_\ast E(n)_\ast$-comodules. 
Furthermore there is a short exact sequence
$$
\xymatrix{
0\longr \kappa(\Laa_n)\rto & Pic(\Laa_n) \rto^-{\mathrm{dim}} &\ZZ \longr 0
}
$$ 
where $\mathrm{dim}(X) = k$.
\end{prop}   

Now let $n=2$ and $p=3$. We then have the following result of
Kamiya and Shimomora \cite{shiPic}.

\begin{prop}\label{ks-l-2} Let $p=3$. There is a non-trivial element of
order $3$ in $\kappa(\cL_2)$ and $\kappa(\cL_2)$ is contained
in $\ZZ/3 \oplus \ZZ/3$; that is, there are inclusions
$$
\Z/3\subset\kappa(\Laa_2)\subset \Z/3\oplus \Z/3\ .  
$$  
\end{prop}

We can now sharpen this result. With Proposition \ref{pic-l-n} the
next result completes the calculation of $\Pic(\Laa_2)$.

\begin{thm}\label{picla2} Let $p=3$. Then 
$\kappa(\Laa_2)$ maps isomorphically to $\kappa_2$.  
\end{thm}
\smallskip

\def\LKM{{\Laa_{n-1}\times_{\Laa_{n-1}\Ka_n}\Ka_n}}

The proof will be supplied below after some preliminaries.

\begin{rem}\label{proof-outline} We will show that for every $X \in 
\kappa_2$  there is an element $Z \in \Pic(\Laa_2)$ so that $L_{K(n)}Z 
\simeq X$. By Theorem \ref{pnot2}, $(E_2)_\ast (-)$ also detects 
dimensions of spheres.  This fact and Proposition \ref{pic-l-n} imply 
that $Z \in  \kappa(\Laa_2)$. 
We can then conclude that the map $\kappa(\cL_2)
\to \kappa_2$  of Lemma \ref{hom-on-pic} is surjective;
with Proposition \ref{ks-l-2}, this will imply Theorem \ref{picla2}.
\end{rem}

We define $\LKM$ to be the category of triples $(X,Y,f)$ with
$X\in \Laa_{n-1}$, $Y\in \Ka_n$, and $f$ a homotopy class of maps
$f:X\to L_{n-1}Y$. A morphism from $(X_1,Y_1,f_1) \to 
(X_2,Y_2,f_2)$ are homotopy classes of morphism
$X_1 \to X_2$ and $Y_1 \to Y_2$ so that 
$$
\xymatrix{
X_1 \rto^{f_1} \dto & Y_1 \dto\\
X_2 \rto_{f_1} & Y_2
}
$$
commutes up to homotopy. We give $\LKM$ an internal smash product
with
$$
(X_1,Y_1,f_1) \wedge (X_2,Y_2,f_2) = 
(X_1 \wedge X_2,L_{K(n)} (Y_1 \wedge Y_2),g)
$$
where $g$ is defined by the following diagram
$$
\xymatrix{
X_1 \wedge X_2 \rto^-{f_1 \wedge f_2}
& L_{n-1}Y_1 \wedge L_{n-1}Y_2\\
& L_{n-1}(Y_1 \wedge Y_2) \rto \ar@{{<}-{>}}[u]_{\simeq} &
L_{n-1}(L_{K(n)} (Y_1 \wedge Y_2)).
}
$$
We will write $f_1 \ast f_2$ for $g$.

If $Z$ is any $L_n$-local spectrum, there is
a {\it chromatic fracture square}:
\begin{equation}\label{chrom-frac}
\xymatrix@C=40pt{
Z \rto^{\eta_Z} \dto & L_{K(n)} Z\dto\\
L_{n-1}Z \rto_-{L_{n-1}\eta_Z} & L_{n-1}L_nZ
}
\end{equation}
where the vertical maps and $\eta_Z$ are the obvious localization maps.
By Theorem 6.19 of \cite{HS} this is a homotopy pull-back square.

\begin{prop}\label{frac-app} The functor
\begin{align*}
\Laa_n&\longr \LKM\\
\\
Z&\longmapsto \big(L_{n-1}Z,L_{K(n)}Z,L_{n-1}(\eta_Z)\big)
\end{align*}
is bijective on isomorphism classes of objects and commutes with
smash products.   
\end{prop}

\begin{proof} The chromatic fracture square (\ref{chrom-frac}) 
gives that this functor is one-to-one on isomorphism classes of
objects. To show that it is onto, let $(X,Y,f)$ be an object
in the target and consider the homotopy pull-back square
$$
\xymatrix{
Z \rto^g \dto & Y \dto \\
X \rto_-f & L_{n-1}Y.
}
$$
Another application of Theorem 6.19 of \cite{HS} shows that any $L_{n-1}$-local
spectrum is $K(n)$-acyclic. Thus $g:Z \to Y$ is $K(n)$-localization. 
Since the map $Y \to L_{n-1}Y$ is an $E(n-1)_\ast$ isomorphism,
so is $Z \to X$; hence $f$ is homotopic to $L_{n-1}g$.

That the functor commutes with smash product is a matter
of definitions.
\end{proof}

\begin{rem}\label{rewrite-pic} By Proposition \ref{frac-app}, the
Picard group $\Pic(\cL_n)$ gets identified with the set of isomorphism classes of triples
$(X,Y,f)$ such that
\begin{enumerate}

\item $X \in \Pic(\cL_{n-1})$ with inverse $X_1$;

\item $Y \in \Pic_n$ with inverse $Y_1$; and,

\item there is a map $f_1:X_1 \to L_{n-1}Y_1$ so that
$f\ast f_1 = L_{n-1}\eta: L_{n-1}S^0 \to L_{n-1}L_{K(n)}S^0$. 
\end{enumerate}

We will be interested in the case where $Y=L_1S^0$.
The next result is Theorem 5.10 of \cite{GHM}.
\end{rem}

\begin{thm}\label{chrom-split-1-redux} Let $X \in \kappa_2$. Then the
localized Hurewicz homomorphism
$$
\pi_0L_1X \longr \pi_0L_1L_{K(2)}(E_2 \wedge X)
$$
is injective. Any choice
of isomorphism $f:(E_2)_\ast \to (E_2)_\ast X$ 
of twisted $\GG_2$-modules 
defines a generator
of $\pi_0L_1X \cong \ZZ_3$. This generator determines
a weak equivalence of $L_1L_{K(2)}S^0$-modules
$$
L_1L_{K(2)}S^0 \simeq L_1X.
$$
\end{thm}

{\bf The Proof of Theorem \ref{picla2}.}
Let $X \in \kappa_2$ and choose an isomorphism
$\phi:(E_2)_\ast \to (E_2)_\ast X$ of
twisted $\GG_2$-modules. Theorem 
\ref{chrom-split-1-redux} now gives an object
$$
(L_1S^0,X,f) \in \LKM
$$
with $f$ the composition
$$
\xymatrix{
L_1S^0 \rto^-{L_1\eta} & L_1L_{K(2)}S^0 \rto &L_1 X.
}
$$
To construct an inverse of $(L_1S^0,X,f)$, let $Y \in \kappa_2$
be an inverse for $X$ and let $\psi: (E_2)_\ast \to (E_2)_\ast Y$
be the isomorphism of twisted $\GG_2$-modules
determined requiring the following composition to be the identity
$$
\xymatrix{
(E_2)_\ast \cong (E_2)_\ast \otimes_{(E_2)_\ast} (E_2)_\ast
\rto^-{\phi \otimes \psi}_-{\cong} &
(E_2)_\ast X \otimes_{(E_2)_\ast} (E_2)_\ast Y \rto^\cong &
(E_2)_\ast (X \wedge Y) \cong (E_2)_\ast.
}
$$
Construct $g:L_1S^0 \to L_1Y$ as we constructed the map $f$.
Then the resulting triple $(L_1S^0,Y,g)$ is the needed inverse. 
For this we need to check that the map
$$
f \ast g: L_1S^0 \to L_1L_{K(2)}(X \wedge Y) \simeq L_1L_{K(2)}S^0
$$
is the standard map $L_1\eta$. But this follows from the choice of $\psi$.

\bibliographystyle{amsplain}
\bibliographystyle{amsplain}
\bibliography{bibghm}
\bigbreak
\bigbreak

\bigbreak

\end{document}